\documentclass[11pt,a4paper]{amsart}
\usepackage[T1]{fontenc}
\usepackage[utf8]{inputenc}
\usepackage{lmodern}
\usepackage{amsmath,amssymb,amsthm,amscd,mathtools,mathrsfs}
\usepackage[a4paper,margin=28mm]{geometry}
\usepackage{microtype,needspace}
\usepackage[hidelinks]{hyperref}
\allowdisplaybreaks[1]
\numberwithin{equation}{section}
\newtheorem{theorem}{Theorem}[section]
\newtheorem{proposition}[theorem]{Proposition}
\newtheorem{lemma}[theorem]{Lemma}
\newtheorem{corollary}[theorem]{Corollary}
\theoremstyle{definition}
\newtheorem{definition}[theorem]{Definition}
\theoremstyle{remark}
\newtheorem{remark}[theorem]{Remark}

\DeclareMathOperator{\Gal}{Gal}

\DeclareMathOperator{\PrePer}{PrePer}
\DeclareMathOperator{\End}{End}
\DeclareMathOperator{\Aut}{Aut}
\DeclareMathOperator{\GL}{GL}
\DeclareMathOperator{\SL}{SL}
\DeclareMathOperator{\PGL}{PGL}
\DeclareMathOperator{\Rat}{Rat}

\newcommand{\Q}{\mathbb Q}
\newcommand{\Qbar}{\overline{\mathbb Q}}
\newcommand{\C}{\mathbb C}
\newcommand{\Z}{\mathbb Z}
\renewcommand{\P}{\mathbb{P}}
\newcommand{\Kbar}{\overline K}
\newcommand{\Kab}{K^{\mathrm{ab}}}
\newcommand{\Kcyc}{K^{\mathrm{cyc}}}

\newcommand{\an}{\mathrm{an}}

\newcommand{\tors}{\mathrm{tors}}
\newcommand{\id}{\mathrm{id}}
\newcommand{\abs}[1]{\lvert #1\rvert}
\newcommand{\norm}[1]{\lVert #1\rVert}
\newcommand{\iter}[2]{#1^{\circ #2}}
\newcommand{\cent}[1]{\mathcal C(#1)}
\newcommand{\centinf}[1]{\mathcal C_\infty(#1)}
\newcommand{\hh}[1]{\widehat h_{#1}}
\newcommand{\lattes}{Latt\`es}
\newcommand{\splitf}{\mathcal F}

\title[$\Kab$ has the Bogomolov property]{$\Kab$ has the Bogomolov property for canonical heights}
\author{Andrea Ferraguti}
\thanks{Andrea Ferraguti: \href{mailto:andrea.ferraguti@unito.it}{\texttt{andrea.ferraguti@unito.it}}}
\author{Carlo Pagano}
\thanks{Carlo Pagano: \href{mailto:carlopagano@google.com}{\texttt{carlopagano@google.com}}}
\makeatletter
\renewcommand{\@setauthors}{%
  \begingroup
  \trivlist
  \centering\normalfont\normalsize
  \@topsep30\p@\relax
  \advance\@topsep by -\baselineskip
  \item\relax
  \begin{minipage}[t]{0.48\textwidth}
    \centering
    \textsc{Andrea Ferraguti}\par\nobreak\vspace{4pt}
    {\small \emph{Universit\`a degli Studi di Torino}\par}
  \end{minipage}\hfill
  \begin{minipage}[t]{0.48\textwidth}
    \centering
    \textsc{Carlo Pagano}\par\nobreak\vspace{4pt}
    {\small \emph{Concordia University}\\\emph{Google DeepMind}\par}
  \end{minipage}
  \endtrivlist
  \endgroup
}
\makeatother
\date{September 2026}

\begin{document}
\begin{abstract}
We show that for any number field $K$ and any rational map $f$ in $K(x)$ of degree at least $2$ that is not conjugate to a power, (signed) Chebyshev or CM Latt\`es map, $K^{\mathrm{ab}}$ has the strong Bogomolov property for the canonical height of $f$. Combined with former results of Amoroso-Zannier, Pottmeyer and Ratazzi, this implies that $\Kab$ has the Bogomolov property for canonical heights of any rational map of degree at least $2$. We also classify the pairs $(f,\alpha)$ whose backward orbit contains infinitely many abelian points. This
settles the Andrews--Petsche conjecture in the greater generality of rational maps over a number field and infinite abelian subsets of backward orbits. 

The authors were led to the main idea of the proof in conversation with \emph{Astra}. The key insight consists of applying the equidistribution results \cite{Yua08}, followed by a classification of $(f,f)$-preperiodic curves \cite{Pak23,Pak20, Bea25} followed by an additional application of equidistribution on a parametrized preperiodic curve. 
\end{abstract}
\maketitle

\section{Introduction and statements}\label{sec:intro}
The purpose of this work is to establish the Bogomolov property on $\Kab$ for canonical heights of rational maps, extending seminal works of Amoroso--Dvornicich and Amoroso--Zannier \cite{AD00, AZ00}, which dealt with the case of the usual Weil height respectively for the rational numbers and for general number fields. 

To state our result more precisely, let us fix some notation. Let $K$ be a number field, let $\Kbar$ be an algebraic closure, and let
$\Kab$ be its maximal abelian extension. Write $G_K=\Gal(\Kbar/K)$. Let $f\in K(z)$ be a map of degree $d \geq 2$. Two pairs $(f,\alpha)$ and $(g,\beta)$ are
\emph{$L$-conjugate} if
\[
 g=m\circ f\circ m^{-1},\qquad \beta=m(\alpha)
 \quad\text{for some }m\in\PGL_2(L).
\]
We use the same terminology for the maps alone. Set
\[
 \PrePer(f)=\{P\in\P^1(\Kbar):\{\iter f n(P):n\geq0\}
                                      \text{ is finite}\},
 \qquad
 f^{-\infty}(\alpha)=\bigcup_{n\geq0}f^{-n}(\alpha).
\]
We denote by $C_d$ the Chebyshev polynomial of degree $d$, normalized by
$C_d(z+z^{-1})=z^d+z^{-d}$. We also recall that $f$ is a \emph{Latt\`es map} if there exist an elliptic curve $E/\overline{\Q}$, a non-constant morphism $\theta\colon E\to E$ and a non-constant map $\pi\colon E\to \P^1$ such that
$$\pi\circ \theta=f\circ \pi.$$
One can show that two elliptic curves giving rise to the same Latt\`es map are isogenous; we therefore define \emph{CM Latt\`es maps} and \emph{non-CM Latt\`es maps} according to an elliptic curve defining them having or not complex multiplication.

\begin{definition}
We say that a rational map $f\in K(z)$ of degree $d$ is \emph{exceptional} if it is
$\Kbar$-conjugate to $z^{\pm d}$, $\pm C_d$ or a Latt\`es map. We say that $f$ is \emph{strongly exceptional} if it is $\Kbar$-conjugate to $z^{\pm d}$, $\pm C_d$ or a CM Latt\`es map.
\end{definition}
Our main result goes as follows. 
\begin{theorem}\label{thm:small}
Let $f\in K(z)$ have degree $d\ge 2$. If $f$ is not strongly exceptional, there is an $\varepsilon>0$ such that
\begin{equation}\label{eq:small-finite}
 \{P\in\P^1(\Kab):\hh f(P)<\varepsilon\}
 \quad\text{is finite}.
\end{equation}
In particular, $\PrePer(f)\cap\P^1(\Kab)$ is finite.
\end{theorem}
Here $\hh f$ is the canonical height associated to $f$; the second assertion of the theorem follows from the fact that preperiodic points are exactly those of canonical height $0$.

An algebraic extension $K/\Q$ is said to have the \emph{Bogomolov property} with respect to a height function $h$ if there exists a positive $\varepsilon$ such that the set
$$\{x\in K: h(x)<\varepsilon\}\setminus \{x\in K: h(x)=0\}$$
is finite. Following \cite{FM15}, we say that $K/\Q$ has the \emph{strong Bogomolov property} with respect to a height function $h$ if there exists a positive $\varepsilon$ such that the set
$$\{x\in K: h(x)<\varepsilon\}$$
is finite. Theorem \ref{thm:small} says that if $f$ is not strongly exceptional, then $\Kab$ has the strong Bogomolov property for $\hh f$. We remark that the results of Amoroso-Zannier \cite{AZ00}, Pottmeyer \cite{Pot13}, and Ratazzi \cite{Rat04} imply that $\Kab$ has the Bogomolov property for $x^{\pm d}$, $\pm C_d$ and CM Latt\`es maps, respectively. We therefore immediately obtain the following corollary.

\begin{corollary}
    Let $f\in K(z)$ be a rational map of degree $\ge 2$. Then $\Kab$ has the Bogomolov property for $\hh f$, and it has the strong Bogomolov property for $\hh f$ if $f$ is not strongly exceptional.
\end{corollary}
Notice that this is sharp in the sense that for every strongly exceptional map there exists a number field $K$ such that $\Kab$ contains infinitely many points of canonical height $0$.

Since
$\hh f(P)=d^{-n}\hh f(\alpha)$ for every $P\in f^{-n}(\alpha)$, Theorem \ref{thm:small} also gives
finiteness of the abelian points in every backward orbit of a
nonexceptional map and in particular unlocks the study of \emph{abelian arboreal representations}. These have been the subject of a conjecture of Andrews--Petsche
\cite[Conjecture~1]{AP20}, which we recall below and settle in the present work in a strong form. 

The Latt\`es case requires an additional definition to pinpoint the abelian class as being basically the one coming from CM-elliptic curves.

\begin{definition}
A pair $(f,\alpha)$, with $f\in K(z)$ a Latt\`es map of degree $d$ and
$\alpha\in\P^1(K)$, is of \emph{CM type} if there are an elliptic curve
$E/K$, an imaginary quadratic field $F\subseteq K$, an endomorphism $a\in\End_{\Kbar}(E)$ of degree
$d$, and a finite morphism $\pi:E\to\P^1$ defined over $\Kab$ such
that
\[
 \End_{\Kbar}(E)\otimes \Q=F,\qquad \alpha\in\pi(E_{\tors}),
\]
and
$$
 f\circ\pi=\pi\circ a.
$$
\end{definition}

The following result extends~\cite[Theorem~D]{FOZ24}.
\begin{theorem}\label{thm:abelian_lattes}\label{thm:general-lattes}
Let $f\in K(z)$ be a Latt\`es map and let $\alpha\in\P^1(K)$ be
nonexceptional for $f$. The following are equivalent:
\begin{enumerate}
\item[\textup{(i)}] $f^{-\infty}(\alpha)\cap\P^1(\Kab)$ is infinite;
\item[\textup{(ii)}] $f^{-\infty}(\alpha)\subseteq\P^1(\Kab)$;
\item[\textup{(iii)}] $(f,\alpha)$ is of CM type.
\end{enumerate}
\end{theorem}

Combining these results with the power and Chebyshev cases gives the
following extension of the Andrews--Petsche conjecture
\cite[Conjecture~1]{AP20}. Recall that a point $\alpha$ is \emph{exceptional for $f$} if its backward orbit is
finite. 
\begin{theorem}\label{general_ap}
Let $f\in K(z)$ have degree $d\geq2$, and let $\alpha\in\P^1(K)$ be
nonexceptional for $f$. Then $f^{-\infty}(\alpha)\cap\P^1(\Kab)$ is
infinite if and only if one of the following holds:
\begin{enumerate}
\item[\textup{(i)}] $(f,\alpha)$ is $\Kab$-conjugate to $(z^{\pm d},\zeta)$,
where $\zeta$ is a root of unity;
\item[\textup{(ii)}] $(f,\alpha)$ is $\Kab$-conjugate to
$(\pm C_d,\zeta+\zeta^{-1})$, where $\zeta$ is a root of unity;
\item[\textup{(iii)}] $f$ is a Latt\`es map and $(f,\alpha)$ is of CM type.
\end{enumerate}
\end{theorem}

Andrews--Petsche~\cite[Conjecture~1]{AP20} proposed the classification of abelian arboreal representations over number fields, in the polynomial case, as essentially saying that they all originated from power and Chebyshev polynomials. They proved it for stable quadratic polynomials over $\Q$ and
for maps geometrically conjugate to powers or Chebyshev polynomials
\cite[Theorems~1, 12 and~13]{AP20}. The authors of this paper removed the stability assumption in the quadratic case over $\Q$ and proved that abelianity forces quadratic polynomials to be postcritically finite \cite[Theorems~1.2--1.3]{FP20} and settled the conjecture for quadratic polynomials over $\Q$. They later treated unicritical polynomials with periodic critical point over arbitrary number fields, and monic unicritical polynomials over number fields of degree at most two~\cite[Theorems~1.2--1.3]{FP23}.

The first author, jointly with Ostafe and Zannier~\cite{FOZ24}, established the reduction to PCF maps in general and proved the Andrews--Petsche conjecture in a strong form over $\Q$. 

Leung--Petsche~\cite[Theorem~1]{LP25} proved
the conjecture whenever the base point is not preperiodic. 

Looper~\cite[Theorem~1.5]{Loo21} proved the strong Bogomolov property for polynomials with a finite superattracting periodic point and a nonarchimedean place without potential good reduction. Her proof combines equidistribution, the geometry of the filled Julia set, and
local restrictions on conjugates in abelian extensions, which will feature also in a crucial way in this work. 

Dvornicich--Zannier~\cite[Theorem~2]{DZ07} proved similar results in the more restrictive generality of \emph{cyclotomic points}.

While preparing this work we have been informed by Zhuchao Ji, Jiarui Song and Junyi Xie \cite{JSX26} about their work also establishing Andrews--Petsche's conjecture. We are grateful to them for correspondence on our respective works and for making it possible to post our works simultaneously.

\subsection*{Structure of the proof}
We begin with a very high level summary. The proof is divided into two parts: the first deals with nonexceptional maps, the second one with non-CM Latt\`es maps. The logic is very similar, but a crucial ingredient is different in the two cases.

The first step of the proof is similar to what happens in \cite{Loo21}. One finds that equidistribution theorems \cite{Yua08} applied to the map $(f,f)$ and to points $(P_n,\sigma_p(P_n))$, for a suitable auxiliary prime $p$ and automorphism $\sigma_p$, tend to conflict with the congruences one has from \cite{AZ00}. The only way this conflict is resolved is in the nonexceptional case that the sequence of points is trapped in an $(f,f)$-periodic curve in $\P^1\times\P^1$, which we can parametrize using \cite{Pak23,Pak20, Bea25}, and in the non-CM Latt\`es case that suitable lifts of these points lie in a torsion translate of an abelian subvariety of $E\times E$ for some elliptic curve $E$, thanks to a theorem of Bogomolov \cite{Zha98}. In both cases, we then re-apply equidistribution to this curve and conclude. 

We now provide some more details on how this logic unfolds in the nonexceptional case. Let $f$ be a nonexceptional map, suppose there is a sequence $\{P_n\}_{n \geq 1}$ of distinct abelian points whose canonical heights tend to zero. We begin with the key observation in ~\cite{AZ00}, the fact that $P_n$ is in $K^{\text{ab}}$ provides a strong congruence between $P_n^{p}$ and a conjugate $\sigma(P_n)$ modulo \emph{all} primes above a suitable prime $p$ of good reduction: this is the key property that allows ~\cite{AZ00} to show the Bogomolov property for the usual Weil height on $K^{\text{ab}}$. Furthermore the congruence is preserved under iteration. On the other hand equidistribution of the small points $(P_n,\sigma(P_n))$ forbids these congruences unless infinitely many pairs satisfy a fixed algebraic relation~\cite{Yua08}. In the latter case one would obtain an $(f,f)$-stable curve containing infinitely many pairs and hence its forward images under $(f,f)$ must eventually repeat.

We then parametrize the resulting periodic curve by $\P^{1}$, using the curve classification~\cite{Pak23,Pak20,Bea25}. A feature of this classification is that it allows us to choose the prime so that essentially the congruence yields a new nontrivial constraint \emph{also when} restricted to the parametrized curve. We apply ~\cite{Yua08} \emph{once more} on this copy of $\P^{1}$ and get a contradiction. It is at this stage of the proof, in the second application of ~\cite{Yua08} that we use that $f$ is nonexceptional. 

\subsection*{Human--AI collaboration} We owe to \emph{Astra} the key insight to blend equidistribution results \cite{Yua08} with the classification of periodic curves \cite{Pak23,Pak20,Bea25} in the way outlined above. We used \emph{Astra}, \emph{Fable} and internal agents at Google DeepMind to draft and review this work. 

\subsection*{Acknowledgments} The first author is grateful to Alina Ostafe and Umberto Zannier for several conversations about these problems, and has been partially supported by the ``National Group for Algebraic and Geometric Structures, and their Applications" (GNSAGA - INdAM). The second author is grateful to Nicole Looper for past conversations surrounding \cite{Loo21} and the \emph{Superhuman reasoning team} at Google DeepMind for several conversations regarding the use of artificial intelligence in mathematics. 

\section{Preliminaries}\label{sec:prelim}

Write $h$ for the absolute logarithmic Weil height. For $f\in K(z)$
of degree $d\geq2$, the canonical height
$\hh f(P)=\lim_{n\to\infty}d^{-n}h(\iter f n(P))$ satisfies
\cite[Chapter~3]{Sil07}
\begin{equation}\label{eq:canonical}
 \hh f\geq0,\qquad \hh f\circ f=d\hh f,\qquad
 \hh f=h+O_f(1),\qquad
 \hh f(P)=0\Longleftrightarrow P\in\PrePer(f).
\end{equation}
It is $G_K$-invariant. For a nonconstant $u\in\Kbar(z)$,
$h(u(P))=(\deg u)h(P)+O_u(1)$.

At each finite place $v$, fix $\Kbar\hookrightarrow\C_v$, where
$\C_v$ is the completed algebraic closure of $K_v$; a tilde denotes reduction.
Put $X_m=(\P^1)^m$, $m\in\{1,2\}$, and
$X_m^{\an}=(X_m\otimes_K\C_v)^{\an}$.
An affine chart of this Berkovich space consists of multiplicative
seminorms $|\cdot|_\zeta$ on $\C_v[T_1,\ldots,T_m]$ extending
$|\cdot|_v$, with the topology of pointwise convergence.
The charts glue to a compact Hausdorff space. An ordinary point
$(t_1,\ldots,t_m)$ gives the seminorm
$g\mapsto|g(t_1,\ldots,t_m)|_v$.

For a multihomogeneous form $H$ of multidegree $(a_1,\ldots,a_m)$,
we use the normalized absolute value
\begin{equation}\label{eq:model}
 \norm H_v(P_1,\ldots,P_m)
 =\frac{|H(\mathbf x_1,\ldots,\mathbf x_m)|_v}
 {\prod_{j=1}^m\max(|x_{j,0}|_v,|x_{j,1}|_v)^{a_j}},
 \qquad P_j=[x_{j,0}:x_{j,1}].
\end{equation}
The denominator makes this independent of the choices of homogeneous
coordinates. It extends continuously to $X_m^{\an}$: on a standard
affine chart, if $h_H$ is the dehomogenization of $H$, its value is
$|h_H|_\zeta/\prod_j\max(1,|T_j|_\zeta)^{a_j}$.

The \emph{Gauss point} $\xi_m\in X_m^{\an}$ is defined by the seminorm
\[
 \left|\sum_I c_I T^I\right|_{\xi_m}=\max_I|c_I|_v.
\]
In other words, at this point the absolute value of a polynomial is the
largest absolute value of its coefficients. Since $|T_j|_{\xi_m}=1$,
\begin{equation}\label{eq:gauss-value}
 \norm H_v(\xi_m)=1
 \qquad\text{if $H$ is integral at $v$ and $\widetilde H\ne0$.}
\end{equation}
Indeed, all the coefficients then have absolute value at most one, and
at least one has absolute value one.

Now let $R\in K(z)$ have degree $r\geq2$ and good reduction at $v$
(see Definition~\ref{def:good_red}). On $X_m$, consider the map
$(R,\ldots,R)$ and the height $\sum_{j=1}^m\hh R(P_j)$.
At a place of good reduction, the probability measure appearing in
equidistribution for this map is
\begin{equation}\label{eq:gauss-measure}
 \mu_{(R,\ldots,R),v}=\delta_{\xi_m}.
\end{equation}
Here $\delta_{\xi_m}$ is the \emph{Gauss measure}: it assigns mass one
to the Gauss point $\xi_m$ and mass zero to its complement. Thus, for
every continuous real-valued function $\varphi$ on $X_m^{\an}$,
\[
 \int_{X_m^{\an}}\varphi\,d\delta_{\xi_m}=\varphi(\xi_m).
\]
The description~\eqref{eq:gauss-measure} is the good-reduction case of
the construction in~\cite[Definition~2.4]{CL06}.

A sequence is \emph{generic} if every proper closed subvariety
contains only finitely many terms. Yuan's
theorem~\cite[Theorem~3.7]{Yua08} implies that, for every generic
sequence $P_n\in X_m(\Kbar)$ with
$\sum_j\hh R(P_{n,j})\to0$, the uniform probability measures on the
finite Galois orbits $G_K\cdot P_n$ converge weakly to the Gauss measure.
Explicitly, for every continuous real-valued function $\varphi$,
\[
 \frac{1}{\#(G_K\cdot P_n)}
 \sum_{Q\in G_K\cdot P_n}\varphi(Q)
 \longrightarrow\varphi(\xi_m).
\]
In particular, taking $\varphi=\norm H_v$ with $H$ as
in~\eqref{eq:gauss-value}, these averages tend to one.
The orbits here are those of the whole tuples: each $\gamma\in G_K$
acts simultaneously on all coordinates. The space $X_m^{\an}$ is
throughout the analytification of the algebraic product $X_m$.

\begin{lemma}\label{lem:semiheight}
Let $u,R\in\Kbar(z)$, with $u$ nonconstant, and let $N\geq1$ satisfy
\[
 u\circ R=\iter f N\circ u,
 \qquad \deg R=d^N\geq2.
\]
Then
$$
 \hh f(u(t))=(\deg u)\hh R(t)
 \qquad(t\in\P^1(\Kbar)).
$$
In particular, every lift under $u$ of an $f$-preperiodic point is
$R$-preperiodic.
\end{lemma}
\begin{proof}
Put $D=d^N$. Divide
$h(u(\iter R n(t)))=(\deg u)h(\iter R n(t))+O_u(1)$ by $D^n$
and pass to the limit, using $u\circ\iter R n=\iter f{Nn}\circ u$.
The last assertion follows from~\eqref{eq:canonical}.
\end{proof}

\begin{lemma}\label{lem:core}
Let $X$ be an integral projective variety over $\Kbar$, let
$T:X\to X$ be a morphism, and let
$\mathfrak h:X(\Kbar)\to\mathbb R_{\geq0}$ satisfy $\mathfrak h(Tx)=q\mathfrak h(x)$ for some
$q>1$. Let $\mathcal S\subseteq X(\Kbar)$ satisfy
$T(\mathcal S)\subseteq\mathcal S$. Assume that $\mathcal S$ contains
no generic sequence $\{x_n\}$ with $\mathfrak h(x_n)\to 0$. Then there is a
proper closed subset $Z\subsetneq X$ such that
\[
 T(Z)\subseteq Z,
\]
and every sequence $\{x_n\}$ in $\mathcal S$ such that $\mathfrak h(x_n)\to 0$ is
eventually contained in $Z$.
\end{lemma}
\begin{proof}
If $\dim X=0$, then $\mathfrak h=0$ and the hypothesis forces
$\mathcal S=\varnothing$; take $Z=\varnothing$. Assume $\dim X>0$.
For $j\geq0$ put
\[
 \mathcal S_j=\{x\in\mathcal S:\mathfrak h(x)<q^{-j}\},
 \qquad Z_j=\overline{\mathcal S_j}^{\,\mathrm{Zar}}.
\]
The descending sequence of closed sets $Z_j$ stabilizes, since $X$
is Noetherian. Write $Z_j=Z$ for all $j\geq j_0$. If $Z=X$, each
$\mathcal S_j$ is Zariski dense. Enumerate the proper closed
subvarieties of $X$, which form a countable collection, and choose
$x_j\in\mathcal S_j$ outside the first $j$ of them. This is a generic sequence with
$\mathfrak h(x_j)\to0$, a contradiction. Thus $Z\neq X$.

The relation $T(\mathcal S_{j+1})\subseteq\mathcal S_j$ gives
$T(Z_{j+1})\subseteq Z_j$: the closed set $T^{-1}(Z_j)$ contains
$\mathcal S_{j+1}$ and hence its closure. For $j\geq j_0$ this is
$T(Z)\subseteq Z$. Finally, a sequence with $\mathfrak h(x_n)\to 0$
is eventually in $\mathcal S_{j_0}\subseteq Z$.
\end{proof}

\begin{remark}\label{rem:finite-components}
If $X$ is a surface, $T$ is finite, and $Z\subsetneq X$ is a closed subset 
with $T(Z)\subseteq Z$, every component of $Z$ consisting of a curve is in fact a preperiodic curve. Indeed, a finite morphism sends a curve onto a curve, and its image
is still a component of $Z$ that is a curve. The finiteness of the set of components gives the desired conclusion. 
\end{remark}

\section{Good reduction and equidistribution}\label{sec:equid}

Fix a finite place $v$. Write $\mathcal O_v$ for the valuation ring of
$K_v$, and use a tilde for reduction. For $\mathbf x=(x_0,x_1)$,
put $\norm{\mathbf x}_v=\max(\abs{x_0}_v,\abs{x_1}_v)$. The chordal
distance is
$$
 \delta_v([x_0:x_1],[y_0:y_1])
 =\frac{\abs{x_0y_1-x_1y_0}_v}
 {\norm{\mathbf x}_v\norm{\mathbf y}_v}.
$$

\begin{definition}\label{def:good_red}
    A degree-$e$ rational map $u\in K(z)$ has \emph{good reduction at $v$} if it has a
homogeneous lift $[U_0:U_1]$ with coefficients in $\mathcal O_v$ and
unit resultant.
\end{definition}
 If $u$ has good reduction, its reduction modulo the maximal ideal of $\mathcal O_v$ has degree $e$, and for every $\mathbf t\in\C_v^2$,
\begin{equation}\label{eq:good-norm}
 \max\{\abs{U_0(\mathbf t)}_v,\abs{U_1(\mathbf t)}_v\}
 =\norm{\mathbf t}_v^e.
\end{equation}
The next lemma contains the key equidistribution obstruction that we will use to prove Theorem \ref{small_nonex}. We remark that this lemma and Lemma \ref{lem:lbg-obstruction}, that is its counterpart for products of elliptic curves, are particular instances of a much more general fact that holds for smooth, projective, geometrically integral varieties. However we decided to state and prove two simpler lemmas, one for each setting, in order to avoid unnecessary technicalities.
\begin{lemma}\label{lem:obstruction}
Let $R\in K(z)$ have degree at least two and good reduction at $v$,
and let $H$ be an integral
multihomogeneous form on $X_m$, $m\in\{1,2\}$, with
$\widetilde H\neq0$. Fix $0<\rho<1$. There is no generic sequence
$\{P_n\}_{n\in \mathbb N}\subseteq X_m(\Kbar)$ satisfying
$$
 \sum_{j=1}^m\hh R(P_{n,j})\longrightarrow0,
 \qquad
 \norm H_v(\gamma P_n)\leq\rho
 \quad(n\geq1,\ \gamma\in G_K).
$$
For $m=1$, every sequence of distinct points is generic. A set of
coordinatewise preperiodic points satisfying the same norm bound is
not Zariski dense; for $m=1$ it is finite.
\end{lemma}
\begin{proof}
The orbit averages of $\norm H_v$ are at most $\rho$.
Equidistribution and~\eqref{eq:gauss-value} make them tend to one,
a contradiction. A Zariski-dense set over the countable field
$\Kbar$ contains a generic sequence, so the same argument applies
to coordinatewise preperiodic points.
\end{proof}

For a rational prime $p$, put
$\Phi_p([x_0:x_1])=[x_0^p:x_1^p]$ and
\begin{equation}\label{eq:section}
 \mathscr S_p(\mathbf X,\mathbf Y)
 =(Y_0X_1^p-Y_1X_0^p)(Y_0^pX_1^p-Y_1^pX_0^p).
\end{equation}
This form has bidegree $(2p,p+1)$ and
\begin{equation}\label{eq:section-norm}
 \norm{\mathscr S_p}_v(P,Q)
 =\delta_v(Q,\Phi_p(P))\,
  \delta_v(\Phi_p(Q),\Phi_p(P)).
\end{equation}
In residue characteristic $p$,
$$
 \widetilde{\mathscr S_p}
 =(Y_0X_1^p-Y_1X_0^p)(Y_0X_1-Y_1X_0)^p\neq0.
$$
The zero locus of its reduction is the union of the Frobenius graph and the diagonal
$\Delta$. For diagonal pairs we also use
\begin{equation}\label{eq:Bp}
 B_p(X_0,X_1)=X_0X_1^p-X_1X_0^p,
 \qquad \norm{B_p}_v(P)=\delta_v(P,\Phi_p(P)),
\end{equation}
whose reduction is nonzero.

For the rest of this section, let $f\in\Qbar(z)$ be nonexceptional
of degree $d$. For $B\in\C(z)$ of degree at least two, set
\[
 \cent B=\{U\in\C(z):\deg U\geq1,\ U\circ B=B\circ U\},
 \qquad
 \centinf B=\bigcup_{n\geq1}\cent{\iter B n}.
\]

\begin{proposition}\label{prop:uniform}
There are $N\geq1$ and nonconstant maps
$F,A_1,\ldots,A_r\in\Qbar(z)$, with $\deg F=d^N$, such that
\begin{equation}\label{eq:uniform-semi}
 \iter f N\circ A_i=A_i\circ F\qquad(1\leq i\leq r),
\end{equation}
and every irreducible $(f,f)$-periodic curve $C\subset(\P^1)^2$
dominating both factors admits a surjective parametrization of one
of the forms
\begin{equation}\label{eq:families}
 (u,w)=(A_i\circ\iter F k,A_j)
 \quad\text{or}\quad
 (u,w)=(A_j,A_i\circ\iter F k),\qquad k\geq0.
\end{equation}
These maps satisfy
\begin{equation}\label{eq:param-semi}
 u\circ F=\iter f N\circ u,
 \qquad w\circ F=\iter f N\circ w.
\end{equation}
\end{proposition}
\begin{proof}
The covering construction in~\cite[Theorems~4.14--4.15 and
Remark~4.16]{Pak23} gives fixed maps $\theta,B\in\Qbar(z)$ with
$\theta$ nonconstant and $B$ not a generalized Latt\`es map, such that
\begin{equation}\label{eq:cover}
 f\circ\theta=\theta\circ B.
\end{equation}
The algebraicity of this construction is proved in
\cite[Section~5.2, proof of Theorem~1.3]{Pak23}.
Use the same cover in both coordinates. By~\cite[Lemma~4.4]{Pak23},
each curve $C$ lifts to a periodic $(B,B)$-curve. Theorem~4.10
there gives a parametrization of $C$ by
\begin{equation}\label{eq:pak-input}
 t\longmapsto(\theta(U(t)),\theta(V(t))),
 \qquad U,V\in\centinf B.
\end{equation}
In particular, $B$ is nonexceptional and $\deg B=d$.

Beaumont's stabilization theorem~\cite[Corollary~1.4]{Bea25} gives
$N\geq1$ with
\begin{equation}\label{eq:stabilization}
 \centinf B=\cent{\iter B N}.
\end{equation}
Put $F=\iter B N$. It is nonexceptional, so
\cite[Theorem~1.2]{Pak20} gives $H_1,\ldots,H_r\in\cent F$ with
\begin{equation}\label{eq:centralizer-family}
 \cent F=\{H_i\circ\iter F a:1\leq i\leq r,\ a\geq0\}.
\end{equation}
For each degree $e$, this set contains only finitely many maps of
degree $e$. The commuting equations define a locus over $\Qbar$
in the parameter space $\Rat_e$, so its finite set of solutions is
algebraic. Thus $H_i\in\Qbar(z)$; cf.~\cite[Section~3.4, proof of
Proposition~3.6(i), Step~1]{Bea25}.

Set $A_i=\theta\circ H_i$. Equation~\eqref{eq:uniform-semi}
follows from~\eqref{eq:cover} and $H_i\circ F=F\circ H_i$.
By~\eqref{eq:stabilization}--\eqref{eq:centralizer-family}, every
parametrization~\eqref{eq:pak-input} has the form
$(A_i\circ\iter F a,A_j\circ\iter F b)$. Removing the common right
factor $\iter F{\min(a,b)}$, which is surjective, preserves its
image and gives~\eqref{eq:families}. Equation~\eqref{eq:param-semi}
follows from~\eqref{eq:uniform-semi}.
\end{proof}

Enlarge $K$ to contain the coefficients of $f,F,A_1,\ldots,A_r$,
and put $e_i=\deg A_i$ and $D=\deg F=d^N$.
\begin{proposition}\label{prop:prime}
There is a rational prime $p>3$ splitting completely in $K$, and a
place $v\mid p$, such that:
\begin{enumerate}
\item $f,F,A_1,\ldots,A_r$ have good reduction at $v$;
\item $p\nmid D\prod_i e_i$;
\item for all $i,j$ and $k\geq0$,
\begin{equation}\label{eq:separate-family}
 A_i\circ\iter F k\neq A_j
 \quad\Longrightarrow\quad
 \widetilde{A_i\circ\iter F k}\neq\widetilde{A_j}.
\end{equation}
\end{enumerate}
Every pair $(u,w)$ in~\eqref{eq:families} then has good reduction,
coordinate degrees prime to $p$, and
\begin{equation}\label{eq:separate-frob}
 \widetilde w\neq\widetilde{\Phi_p\circ u}.
\end{equation}
If $u\ne w$, then $\widetilde u\neq\widetilde w$.
\end{proposition}
\begin{proof}
The first two conditions exclude finitely many places. Elsewhere,
reduction preserves degrees, so~\eqref{eq:separate-family} can fail
only when $e_iD^k=e_j$. For each $(i,j)$ there is at most one such
$k$. For every such triple with distinct maps, exclude the finitely
many places where the nonzero cross-product of their homogeneous
lifts reduces to zero. Chebotarev's theorem
\cite[Theorem~8.31 and Corollary~8.32]{Mne20} supplies a rational
prime splitting completely in the normal closure of $K/\Q$ and
avoiding these places and the primes at most three.

Composition preserves good reduction. The coordinate degrees in
\eqref{eq:families} are prime to $p$, whereas
$\deg(\Phi_p\circ u)=p\deg u$. This proves~\eqref{eq:separate-frob}.
The last assertion follows from~\eqref{eq:separate-family}.
\end{proof}

\begin{proposition}\label{cor:curve}
Fix $K,p,v$ as above. Let $C\neq\Delta$ be an irreducible
$(f,f)$-periodic curve dominating both factors, and fix $0<\rho<1$.
There is no sequence of distinct $(P_n,Q_n)\in C(\Kbar)$ such that
$$
 \hh f(P_n)+\hh f(Q_n)\longrightarrow0,
 \qquad
 \norm{\mathscr S_p}_v(\gamma P_n,\gamma Q_n)\leq\rho
 \quad(\gamma\in G_K).
$$
In particular, only finitely many coordinatewise preperiodic pairs
on $C$ satisfy this norm bound on all conjugates.
\end{proposition}
\begin{proof}
Take $(u,w)$ as in Proposition~\ref{prop:uniform}, with good
homogeneous lifts $\mathbf U=(U_0,U_1)$ and $\mathbf W=(W_0,W_1)$.
The form $H_C=\mathscr S_p(\mathbf U,\mathbf W)$ is integral,
homogeneous of degree $2p\deg u+(p+1)\deg w$, and
$$
 \widetilde H_C
 =(\widetilde W_0\widetilde U_1^{\,p}
       -\widetilde W_1\widetilde U_0^{\,p})
  (\widetilde W_0\widetilde U_1
       -\widetilde W_1\widetilde U_0)^p\neq0.
$$
The first factor is nonzero by~\eqref{eq:separate-frob}; the second
is nonzero by Proposition~\ref{prop:prime}, since $u\ne w$.

Lift $(P_n,Q_n)$ to $t_n\in\P^1(\Kbar)$. The $t_n$ are distinct,
and Lemma~\ref{lem:semiheight} gives
\[
 \hh F(t_n)=\frac{\hh f(P_n)}{\deg u}\longrightarrow0.
\]
By~\eqref{eq:good-norm} and~\eqref{eq:model},
$\norm{H_C}_v(t)=\norm{\mathscr S_p}_v(u(t),w(t))$.
Since $u,w$ are defined over $K$, the assumed bound gives
$\norm{H_C}_v(\gamma t_n)\leq\rho$ for every $\gamma\in G_K$.
This contradicts Lemma~\ref{lem:obstruction} for $R=F$ on $\P^1$.
\end{proof}

\section{Local abelian congruences and collisions}\label{sec:local}

Fix $K,p,v$ as in Proposition~\ref{prop:prime}, normalize
$\abs p_v=p^{-1}$, and fix $\Kbar\hookrightarrow\C_p$.
For a finite abelian extension $L/K$, let $w$ be the induced place,
let $D_w\subseteq\Gal(L/K)$ be its decomposition group, and set
$e=e(w/v)$. Since $K_v=\Q_p$, the extension $L_w/\Q_p$ is abelian.
We identify $D_w$ with $\Gal(L_w/K_v)$.

\begin{lemma}\label{lem:az}
Let $P\in\P^1(\Kab)$ and $L=K(P)$.
\begin{enumerate}
\item There is $\sigma_P\in D_w$ satisfying one of the following:
\begin{enumerate}
\item[\textup{(U)}] $L_w/K_v$ is unramified and
\[
 \delta_v(\sigma_Px,\Phi_p(x))\leq p^{-1}
 \qquad(x\in\P^1(L));
\]
\item[\textup{(R)}] $L_w/K_v$ is ramified, $\sigma_P\neq\id$, and
\[
 \delta_v(\Phi_p(\sigma_Px),\Phi_p(x))\leq p^{-1}
 \qquad(x\in\P^1(L)).
\]
\end{enumerate}
\item Put $Q_P=\sigma_PP$. For all $n\geq0$ and $\gamma\in G_K$,
\begin{equation}\label{eq:persistent}
 \norm{\mathscr S_p}_v
 (\gamma\iter f n(P),\gamma\iter f n(Q_P))\leq p^{-1}.
\end{equation}
\item If $Q_P=P$, then \textup{(U)} holds and
\begin{equation}\label{eq:diagonalU}
 \norm{B_p}_v(\gamma P)\leq p^{-1}
 \qquad(\gamma\in G_K).
\end{equation}
\end{enumerate}
\end{lemma}
\begin{proof}
Amoroso--Zannier~\cite[Lemmas~3.1--3.2]{AZ00} give $\phi\in D_w$
and a subgroup $H\subseteq D_w$ of order at least $\min\{e,p\}$ such that
\[
 \abs{\phi(a)-a^p}_v\leq p^{-1/e},\qquad
 \abs{(\tau a)^p-a^p}_v\leq p^{-1}
 \quad(\tau\in H,\ a\in\mathcal O_{L_w}).
\]
Their congruences extend from algebraic integers in $L$ to
$\mathcal O_{L_w}$ by density. If $e=1$, take $\sigma_P=\phi$;
otherwise take $\sigma_P\in H\setminus\{\id\}$. To obtain the chordal
inequalities, use the affine coordinate $z$ when $\abs z_v\leq1$
and $z^{-1}$ otherwise. Inversion preserves chordal distance and
commutes with $\Phi_p$ and the automorphisms in $D_w$.

The extension $L/K$ is abelian, so $\gamma P\in\P^1(L)$ and
\[
 \sigma_P\iter f n(\gamma P)=\gamma\iter f n(Q_P).
\]
Apply (1) to $\iter f n(\gamma P)$ and use
\eqref{eq:section-norm}. This proves (2), since both chordal factors
are at most one. In case \textup{(R)}, the nonidentity automorphism
$\sigma_P$ cannot fix the generator $P$ of $L/K$. Thus $Q_P=P$
forces \textup{(U)}, and (3) follows from~\eqref{eq:Bp}.
\end{proof}

\begin{remark}
Equation~\eqref{eq:persistent} uses a congruence on all of $L$,
not a commutation relation between $f$ and $\Phi_p$.
The automorphism $\gamma$ need not preserve $w$: the congruence is
applied to $\iter f n(\gamma P)\in\P^1(L)$ at the fixed place.
\end{remark}

\begin{lemma}\label{lem:close}
If $P,Q\in\P^1(\C_p)$ satisfy
$\delta_v(\Phi_p(P),\Phi_p(Q))\leq p^{-1}$, then
$$
 \delta_v(P,Q)\leq p^{-1/p}<1.
$$
\end{lemma}
\begin{proof}
Frobenius is injective on the residue field, so $P$ and $Q$ have
the same reduction. In a common integral affine chart, write them
as $x,y$ with $\abs x_v,\abs y_v\leq1$. Since $p$ divides every
intermediate binomial coefficient,
\[
 \abs{x-y}_v^p\leq\max\{\abs{x^p-y^p}_v,p^{-1}\}\leq p^{-1}.
\]
In this chart $\delta_v(P,Q)=\abs{x-y}_v$.
\end{proof}

\begin{lemma}\label{lem:collision}
Let $g\in K_v(z)$ have good reduction and separable reduction.
There is an integral homogeneous form $J_g$ of degree
$2\deg g-2$, with $\widetilde J_g\neq0$, such that for all $P,Q\in\P^1(\C_v)$ and $0<r<1$,
\begin{equation}\label{eq:collision}
 P\neq Q,\quad g(P)=g(Q),\quad\delta_v(P,Q)\leq r
 \quad\Longrightarrow\quad\norm{J_g}_v(P)\leq r.
\end{equation}
\end{lemma}
\begin{proof}
Choose a good homogeneous lift $g=[A:B]$ of degree $e$ and set
$$
 D_g(\mathbf X,\mathbf Y)
 =\frac{A(\mathbf X)B(\mathbf Y)-B(\mathbf X)A(\mathbf Y)}
 {X_0Y_1-X_1Y_0}.
$$
The quotient is integral and bihomogeneous of bidegree $(e-1,e-1)$.
Put
$$
 J_g(\mathbf X)=D_g(\mathbf X,\mathbf X).
$$
For $a(z)=A(z,1)$ and $b(z)=B(z,1)$, we have
$J_g(z,1)=a'(z)b(z)-a(z)b'(z)$. Its reduction is nonzero because
$\widetilde g$ is separable.

The points in~\eqref{eq:collision} have the same reduction. Choose
norm-one representatives $\mathbf x,\mathbf y$ in a common affine
chart, with the same coordinate equal to one. Then
$\norm{\mathbf x-\mathbf y}_v=\delta_v(P,Q)$ and
$D_g(\mathbf x,\mathbf y)=0$. Integral polynomials are
$1$-Lipschitz on the unit polydisc, whence
\[
 \norm{J_g}_v(P)
 =\abs{D_g(\mathbf x,\mathbf x)-D_g(\mathbf x,\mathbf y)}_v
 \leq\delta_v(P,Q)\leq r.
\]
\end{proof}

\begin{remark}\label{rem:all-iterates}
Every iterate $g=\iter f m$, $m\geq1$, has good reduction of
degree $d^m$ prime to $p$. Its reduction is therefore separable.
Thus Lemma~\ref{lem:collision} applies to an iterate chosen after
$p$ is fixed.
\end{remark}

\section{Small abelian points for nonexceptional maps}\label{sec:small}
In this section, we prove the following theorem.
\begin{theorem}\label{small_nonex}
    Let $f\in K(z)$ be a nonexceptional map of degree $d\ge 2$. Then there is an $\varepsilon>0$ such that
    $$\{P\in\P^1(\Kab):\hh f(P)<\varepsilon\}
 \quad\text{is finite}.$$
\end{theorem}
\begin{proof}
Enlarge $K$ as required by Proposition~\ref{prop:uniform}; this is
harmless because $K^{\mathrm{ab}}\subseteq(K')^{\mathrm{ab}}$ for
every finite extension $K'/K$. Choose $p,v$ as in
Proposition~\ref{prop:prime}. If~\eqref{eq:small-finite} fails,
there are distinct $P_i\in\P^1(\Kab)$ with $\hh f(P_i)\longrightarrow0$.

Choose $\sigma_{P_i}$ as in Lemma~\ref{lem:az} and put
$Q_i=\sigma_{P_i}P_i$. If $P_i=Q_i$ infinitely often, then
\eqref{eq:diagonalU} contradicts Lemma~\ref{lem:obstruction} with
$R=f$ and $H=B_p$. Discard these indices. We now have
$P_i\neq Q_i$ and $\hh f(Q_i)=\hh f(P_i)$.

Write $\splitf=(f,f)$ and set
\[
 T=\bigcup_{n\geq0}\splitf^{\circ n}\{(P_i,Q_i):i\geq1\},
 \qquad H(P,Q)=\hh f(P)+\hh f(Q).
\]
Then $\splitf(T)\subseteq T$ and $H\circ\splitf=dH$.
Every point of $T$ satisfies~\eqref{eq:persistent} on all its
$K$-conjugates. Lemma~\ref{lem:obstruction}, applied to
$\mathscr S_p$, excludes a generic sequence in $T$ with $H\to0$.
Lemma~\ref{lem:core} therefore gives a proper closed set
$$
 Z\subsetneq(\P^1)^2,\qquad\splitf(Z)\subseteq Z,
$$
containing all but finitely many $(P_i,Q_i)$.

An irreducible curve component $C$ of $Z$ contains infinitely many
of these pairs; restrict to that subsequence. Since $P_i\neq Q_i$,
we have $C\neq\Delta$. The first coordinates are distinct. A fixed
second coordinate has only finitely many $K$-conjugates, so $C$
dominates both factors. By Remark~\ref{rem:finite-components}, $C$
is preperiodic. Choose $m\geq0$ such that
$D=\splitf^{\circ m}(C)$ is periodic, and put $g=\iter f m$.

If $D\neq\Delta$, the finite map $\splitf^{\circ m}$ gives
infinitely many distinct image pairs $(g(P_i),g(Q_i))$ on $D$.
Their heights tend to zero and their all-conjugate norm bound
persists by~\eqref{eq:persistent}, contradicting
Proposition~\ref{cor:curve}.

Suppose $D=\Delta$. Then $m\geq1$ and
\begin{equation}\label{eq:equalimages}
 P_i\neq Q_i,\qquad g(P_i)=g(Q_i).
\end{equation}
Restrict to a subsequence for which the same case of
Lemma~\ref{lem:az}(1) holds.

In case \textup{(U)}, $\sigma_{P_i}$ fixes $g(\gamma P_i)$ for
every $\gamma\in G_K$. Applying \textup{(U)} at this point gives
\[
 \norm{B_p}_v(\gamma g(P_i))\leq p^{-1}.
\]
Since $g$ is finite, the points $g(P_i)$ include a distinct
subsequence, with heights $d^m\hh f(P_i)\to0$. This contradicts
Lemma~\ref{lem:obstruction}.

In case \textup{(R)}, Lemma~\ref{lem:close} gives
$\delta_v(\gamma P_i,\gamma Q_i)\leq p^{-1/p}$ for every
$\gamma\in G_K$. The map $g$ has separable good reduction by
Remark~\ref{rem:all-iterates}. Lemma~\ref{lem:collision}, applied
to all conjugates of~\eqref{eq:equalimages}, yields
$$
 \norm{J_g}_v(\gamma P_i)\leq p^{-1/p}<1,
 \qquad\widetilde J_g\neq0.
$$
This again contradicts Lemma~\ref{lem:obstruction}.
\end{proof}

\section{Small abelian points for non-CM Latt\`es maps}\label{sec:lbg}

In this section we prove the following theorem that, together with Theorem \ref{small_nonex}, proves Theorem \ref{thm:small}.

\begin{theorem}\label{thm:lbg-main}
Let $f\in K(z)$ be a non-CM Latt\`es map of degree $d\ge 2$. Then there is an $\varepsilon>0$ such that
    $$\{P\in\P^1(\Kab):\hh f(P)<\varepsilon\}
 \quad\text{is finite}.$$
\end{theorem}

We start by recalling a few basic facts on Latt\`es maps.
By the classification first proven by Milnor in~\cite[Theorem~3.1]{Mil06},
there are an elliptic curve $E/\Kbar$, a cyclic group
$\Gamma\subseteq\Aut(E,O)$ of order $2$, $3$, $4$, or $6$, and
a quotient map $\pi:E\to E/\Gamma\simeq\P^1$ such that
\begin{equation}\label{eq:lattes}
 f\circ\pi=\pi\circ\psi,
 \qquad \psi(P)=a(P)+b,\qquad \deg a=\deg f=d.
\end{equation}
Write $\widehat h_E$ for the N\'eron--Tate height associated to $(O)$.
The map $\psi$ commutes with a generator $\zeta$ of $\Gamma$, so
$(1-\zeta)b=0$ and $b$ is torsion. Consequently
$$
 \PrePer(f)=\pi(E_{\tors}),
 \qquad \pi^{-1}(\PrePer(f))=E_{\tors}.
$$
Indeed, torsion points have finite $\psi$-orbits. Conversely, a lift of
a preperiodic point has finite orbit because $\pi$ has finite fibers;
the relation $\widehat h_E(\psi(P))=d\widehat h_E(P)$ then forces it
to be torsion. The same finite-fiber argument shows that, for any
finite morphism $\theta$ satisfying $f\circ\theta=\theta\circ a$,
$\theta(P)$ is preperiodic if and only if $P$ is torsion.

For the rest of the section, we assume that $f$ is a non-CM Latt\`es map, so that $E/\overline{K}$ is a non-CM elliptic curve. Serre's open-image
theorem~\cite{Ser72} gives an open subgroup
$G:=\rho_E(G_K)\subseteq\GL_2(\widehat\Z)$. Moreover,
$$
 H:=\rho_E(G_{\Kab})=\overline{[G,G]},
 \qquad I:=[\SL_2(\widehat\Z):H]<\infty;
$$
see~\cite[Theorem~1.1 and Lemmas~2.1(i), 3.3]{Zyw25}.

\begin{proposition}\label{prop:torsion}
For every $P\in E_{\tors}$ of exact order $n$,
$$
 [\Kab(P):\Kab]
 \ge\frac{n^2}{I}\prod_{\ell\mid n}(1-\ell^{-2})
 \ge\frac{n^2}{I\zeta(2)}.
$$
Consequently, every non-CM Latt\`es map $f\in K(z)$ satisfies
$\#(\PrePer(f)\cap\P^1(\Kab))<\infty$.
\end{proposition}
\begin{proof}
The case $n=1$ is immediate. For $n>1$, the points of exact order $n$ in $E[n]\simeq(\Z/n\Z)^2$ are
the primitive vectors. The group $\SL_2(\widehat\Z)$ acts transitively
on them, and their number is
$n^2\prod_{\ell\mid n}(1-\ell^{-2})$. An index-$I$ subgroup has
orbits of size at least $1/I$ of this number. The $H$-orbit of $P$
has size $[\Kab(P):\Kab]$, proving the first inequality; the Euler
product for $\zeta(2)$ proves the second. In particular, only finitely
many torsion points have bounded degree over $\Kab$.

For the consequence, choose a finite extension $L/K$ over which a
normalized presentation of $f$ is defined. If
$\beta=\pi(P)\in\P^1(\Kab)$ is preperiodic, then $P$ is torsion
and its $G_{L^{\mathrm{ab}}}$-orbit lies in a fiber of $\pi$.
Thus $[L^{\mathrm{ab}}(P):L^{\mathrm{ab}}]\le\deg\pi$.
Applying the first assertion over $L$ gives only finitely many such
$P$, hence only finitely many $\beta$.
\end{proof}


For an elliptic curve $E/K$ with origin $O$, we let $\widehat h_E$ be
the N\'eron--Tate height associated with $\mathcal O_E(O)$. Thus,
for a Weierstrass $x$-coordinate,
$$
 h(x(P))=2\widehat h_E(P)+O_E(1).
$$
By Proposition \ref{prop:torsion}, for every $D\ge 1$ we have that
\begin{equation}\label{finitely_many_pts}
    \mathcal T_D(E/K)\coloneqq \{P\in E(\overline{K})_{tors}: [\Kab(P):\Kab]\le D\}\quad \mbox{is finite}.
\end{equation}

We now proceed with a similar strategy to that we used to prove Theorem \ref{small_nonex}. For every rational prime $p>3$ and place $v$ of $K$ with
$K_v=\Q_p$, we normalize $\abs p_v=p^{-1}$, and fix an embedding
$\Kbar\hookrightarrow\C_v$. Put
\[
 \Phi_p([X_0:X_1])=[X_0^p:X_1^p].
\]
We now define the multihomogeneous forms
$$
 \mathscr U_p(\mathbf X,\mathbf Y)=Y_0X_1^p-Y_1X_0^p,
       $$
       $$
 \mathscr R_p(\mathbf X,\mathbf Y)=Y_0^pX_1^p-Y_1^pX_0^p,
$$
that have bidegrees $(p,1)$ and $(p,p)$, respectively, and whose
product is the section $\mathscr S_p$ defined in \eqref{eq:section}. Their standard model norms satisfy
$$
 \norm{\mathscr U_p}_v(A,B)
    =\delta_v(B,\Phi_p(A)),
       $$
       $$
 \norm{\mathscr R_p}_v(A,B)
    =\delta_v(\Phi_p(B),\Phi_p(A)),
$$
and both norms extend continuously to $(\P^1\times\P^1)^{\an}$.

We now prove the analogue of Lemma \ref{lem:obstruction} for an elliptic curve or its square. Recall that a sequence is \emph{generic}
if every proper closed subvariety contains only finitely many terms.

\begin{lemma}
\label{lem:lbg-obstruction}
Let $E$ be an elliptic curve over a number field $F$, with good
reduction at a finite place $w$. Let $r\in\{1,2\}$, put $A=E^r$,
and give $A$ the product N\'eron--Tate height
\[
 \widehat h_A(P_1,\ldots,P_r)=\sum_{j=1}^r\widehat h_E(P_j).
\]
Let $\mathcal A=\mathcal E^r$ be its smooth proper model over
$\mathcal O_w$. Let $\mathcal M$ be a line bundle on $\mathcal A$
and let $s$ be a global section whose restriction to the special
fiber is not identically zero. Equip $\mathcal M$ on the generic
fiber with its model metric. For $0<\rho<1$, there is no generic
sequence $\{Z_n\}_{n\in \mathbb N}\subseteq A(\overline F)$ such that
\[
 \widehat h_A(Z_n)\longrightarrow0,
 \quad \mbox{and} \quad
 \norm s_w(\gamma Z_n)\leq\rho
 \quad \mbox{for every}\quad  n\geq1,\ \gamma\in G_F.
\]
\end{lemma}

\begin{proof}
Let $\xi_{A,w}\in A_w^{\an}$ be the divisorial analytic point
associated with the generic point of the special fiber of
$\mathcal A$. As noted by Chambert-Loir in \cite[Example 3.3]{CL06}, the limiting measure is the Dirac measure $\delta_{\xi_{A,w}}$ concentrated in $\xi_{A,w}$.

The assumption on $s$ gives
\begin{equation}\label{eq:lbg-generic-value}
 \norm s_w(\xi_{A,w})=1.
\end{equation}
Indeed, near the generic point of the special fiber choose a frame
of $\mathcal M$ and write $s=h e$. The element $h$ is a unit in
the local ring at that generic point, since the reduction of $s$
is nonzero. Its vertical valuation is zero, which is precisely
\eqref{eq:lbg-generic-value}.

Multiplication by $2$ on $A$ is polarized by
$\bigotimes_j\mathrm{pr}_j^*\mathcal O_E(O)$, with polarization
factor $4$. Yuan's theorem \cite[Theorem~3.7]{Yua08} gives
\[
 \mu_{Z_n,w}:=
 \frac{1}{\#(G_F\cdot Z_n)}
       \sum_{Z\in G_F\cdot Z_n}\delta_Z
 \ \Longrightarrow\ \delta_{\xi_{A,w}}.
\]
The function $\norm s_w$ is continuous on the compact space
$A_w^{\an}$. Hence its integrals against these orbit measures
tend to $1$ by \eqref{eq:lbg-generic-value}, whereas the
all-conjugate bound makes every integral at most $\rho<1$.
This is a contradiction.
\end{proof}

In applications, the section in this lemma is pulled back from a
product of projective lines. If a morphism
$\varphi:\mathcal A\to(\P^1_{\mathcal O_w})^2$ extends to the
models and $H$ is an integral bihomogeneous section, the pullback
model metric gives the exact identity
\begin{equation}\label{eq:lbg-pullback}
 \norm{\varphi^*H}_w(Z)=\norm H_w(\varphi(Z)).
\end{equation}
Moreover, reduction commutes with this pullback. Thus it is enough
to check that $\widetilde\varphi^*\widetilde H$ is not the zero
section. This formulation does not require homogeneous coordinate
polynomials on the elliptic curve.

\begin{proposition}\label{prop:lbg-elliptic-coordinate}
Let $E/K$ be a non-CM elliptic curve and let $x:E\to\P^1$ be a
Weierstrass quotient by $\{\pm1\}$. There is $c_E=c(E,K)>0$ such
that
\begin{equation}\label{eq:lbg-elliptic-gap}
 \widehat h_E(P)\geq c_E
 \quad\text{if}\quad
 x(P)\in\P^1(\Kab),\qquad P\notin E(\Kbar)_{\tors}.
\end{equation}
Furthermore, only finitely many torsion points $P$ satisfy
$x(P)\in\P^1(\Kab)$.
\end{proposition}

\begin{proof}
Throughout the proof, we fix a rational prime $p$ and a place $v$ of $K$ dividing $p$ with the following properties:
\begin{enumerate}
    \item $p$ splits completely in $K$;
    \item $E$ has good reduction at $v$;
    \item $x$ extends to a finite morphism of degree $2$ from the smooth proper model $\mathcal E/\mathcal O_v$ to $\P^1_{\mathcal O_v}$;
    \item $p$ does not divide the order of any point in $\mathcal T_4(E/K)$;
    \item the reduction map $\mathcal T_4(E/K)\to \widetilde{E}$ is injective. 
\end{enumerate}
This can be done by combining Chebotarev density theorem with \eqref{finitely_many_pts}.
We moreover fix an embedding $\overline{K}\hookrightarrow \C_v$ and normalize the absolute value by $|p|_v=p^{-1}$.

We argue by contradiction. Suppose that \eqref{eq:lbg-elliptic-gap} is false. Choose a sequence $\{P_n\}$ of points of $E$ such that
\begin{equation}\label{eq:lbg-sequence}
 \alpha_n:=x(P_n)\in\Kab,\qquad
 0<\widehat h_E(P_n)\longrightarrow0,
\end{equation}
with the heights strictly decreasing, so that in particular the $\alpha_n$
are distinct.

By Lemma \ref{lem:az}, each $\alpha_n$ has a conjugate $\beta_n$ such that one of the following holds:
\begin{align}
 \textup{(U)}\quad
 \norm{\mathscr U_p}_v(\gamma\alpha_n,\gamma\beta_n)&\leq p^{-1} \quad \mbox{for every }\gamma \in G_K,
        \label{eq:lbg-U-bound}\\
 \textup{(R)}\quad
 \norm{\mathscr R_p}_v(\gamma\alpha_n,\gamma\beta_n)&\leq p^{-1} \quad \mbox{for every }\gamma \in G_K.
        \label{eq:lbg-R-bound}
\end{align}
In particular if \textup{(R)} holds then $\beta_n\ne\alpha_n$.

Now pass to an infinite
subsequence with one fixed label $\lambda\in\{\mathrm U,\mathrm R\}$,
and write $H_p=\mathscr U_p$ or $H_p=\mathscr R_p$ accordingly.
Choose any $Q_n\in E(\Kbar)$ with $x(Q_n)=\beta_n$.
An extension to $\Kbar$ of the automorphism carrying $\alpha_n$
to $\beta_n$ sends $P_n$ to one of $Q_n,-Q_n$. Galois invariance
and evenness of the N\'eron--Tate height give
\begin{equation}\label{eq:lbg-equal-heights}
 \widehat h_E(Q_n)=\widehat h_E(P_n)>0.
\end{equation}
In particular, the $Q_n$ also form a sequence of distinct points.
Since $x$ is defined over $K$ and has degree two,
\begin{equation}\label{eq:lbg-lift-degrees}
 [\Kab(P_n):\Kab]\leq2,
 \qquad [\Kab(Q_n):\Kab]\leq2.
\end{equation}

The morphism
\[
 (x,x):E^2\longrightarrow(\P^1)^2
\]
extends to the good models at $v$, and its reduction is finite
and surjective. Both $\widetilde{\mathscr U_p}$ and
$\widetilde{\mathscr R_p}$ are nonzero sections, so the reduction
of $(x,x)^*H_p$ is nonzero. On every $K$-Galois conjugate of
$Z_n=(P_n,Q_n)$, equations \eqref{eq:lbg-U-bound},
\eqref{eq:lbg-R-bound} and \eqref{eq:lbg-pullback} give
\begin{equation}\label{eq:lbg-pair-bound}
 \norm{(x,x)^*H_p}_v(\gamma Z_n)
 =\norm{H_p}_v(\gamma\alpha_n,\gamma\beta_n)
 \leq p^{-1}.
\end{equation}
The product height of $Z_n$ tends to zero by
\eqref{eq:lbg-sequence}--\eqref{eq:lbg-equal-heights}; if the set of all the $Z_n$ were Zariski dense in $E^2$, it would contain
a generic subsequence with these same properties, contradicting
Lemma~\ref{lem:lbg-obstruction}.\footnote{To obtain such a subsequence one
enumerates the proper closed subvarieties over the countable field
$\Kbar$ and successively avoids their finite unions.}

Thus the Zariski closure of the pairs is proper. One irreducible
curve component $C\subset E^2_{\Kbar}$ contains infinitely many
of them; restrict to this subsequence. The distinctness of both
coordinate sequences implies that $C$ dominates both factors.

The sequence on $C$ is generic and its product N\'eron--Tate height
tends to zero. The classical Bogomolov theorem
\cite[Corollary~3]{Zha98}, applied over a number field defining $C$,
implies
\[
 C=(T_1,T_2)+B,
\]
where $T_1,T_2$ are torsion points and $B\subset E^2$ is a
one-dimensional abelian subvariety. The two projections of $B$
are isogenies onto $E$. Choose an isogeny $E\to B$. Since
$\End_{\Kbar}(E)=\Z$, the two compositions of such an isogeny with the two projections are multiplications by non-zero integers $a$ and $b$. For each $n$, choose $R_n$ with
\[
 P_n=T_1+[a]R_n,
 \qquad Q_n=T_2+[b]R_n.
\]
Equation~\eqref{eq:lbg-equal-heights} then gives:
\[
 a^2\widehat h_E(R_n)=b^2\widehat h_E(R_n).
\]
Since the common height is positive, we have $b=\epsilon a$ for one fixed
$\epsilon\in\{1,-1\}$. Setting $T=T_2-\epsilon T_1$, we obtain
$$
 Q_n=\epsilon P_n+T,
 \qquad T\in E(\Kbar)_{\tors}.
$$
In fact $C$ is the graph of this translated sign map. By
\eqref{eq:lbg-lift-degrees} we thus get
$$
 [\Kab(T):\Kab]
 \leq[\Kab(P_n,Q_n):\Kab]\leq4,
$$
so that $T\in\mathcal T_4(E/K)$. By the choice of $p$, the order of $T$ is prime to $p$ and if $T\ne O$, then $\widetilde T\ne \widetilde O$.

Let $M=K(T)$ and let $w\mid v$ be induced by the same embedding
into $\C_v$, so that $|p|_w=p^{-1}$. Set
\[
 u(P)=x(P),\qquad w_T(P)=x(\epsilon P+T).
\]
By properness $T$ extends to a section of the good elliptic model;
translation by it and multiplication by $\epsilon$ are model
automorphisms. Thus $u,w_T$ extend to degree-two morphisms on the
models, including on the special fiber. For $\gamma\in G_M$,
\[
 \gamma Q_n=\epsilon\gamma P_n+T,
\]
so \eqref{eq:lbg-pair-bound} gives
\begin{equation}\label{eq:lbg-curve-estimate}
 \norm{(u,w_T)^*H_p}_w(\gamma P_n)\leq p^{-1}
 \qquad(\gamma\in G_M).
\end{equation}
We now claim that the reduction of $(u,w_T)^*H_p$ does not vanish in either case.

If \textup{(U)} holds, the reduction of
$(u,w_T)^*\mathscr U_p$ vanishes identically only if $\widetilde w_T=\Phi_p\circ\widetilde u$ as maps of the reduced elliptic curve. But this is impossible: the two degrees are $2$ and $2p$. 

If \textup{(R)} holds, then $\beta_n\ne\alpha_n$ and hence $T\ne O$, since
$x(\epsilon P)=x(P)$. It follows that $\widetilde T\ne O$, by
our choice of $p$. The pole divisors on the reduced
elliptic curve are
\[
 \operatorname{div}_\infty(\widetilde u)=2(O),\qquad
 \operatorname{div}_\infty(\widetilde w_T)
                  =2(-\epsilon\widetilde T),
\]
so that $\widetilde u\ne\widetilde w_T$. On the other hand $\widetilde{\mathscr R_p}$ is the $p$-th power of the diagonal equation. Its pullback is therefore
nonzero as well.

We are therefore in the position to apply Lemma~\ref{lem:lbg-obstruction} to $E/M$
and the distinct points $P_n$. Their heights tend to zero, while
\eqref{eq:lbg-curve-estimate} bounds the norm on all $M$-Galois
conjugates by $p^{-1}<1$. This contradiction proves
\eqref{eq:lbg-elliptic-gap}.

Finally, if $P$ is torsion and $x(P)\in\P^1(\Kab)$, its
$G_{\Kab}$-orbit is contained in the fiber $x^{-1}(x(P))$, of
cardinality at most two. Thus $P\in\mathcal T_2(E/K)$, which is
finite by \eqref{finitely_many_pts}.
\end{proof}

\begin{proof}[Proof of Theorem~\ref{thm:lbg-main}]
Let $f\in K(z)$ be a \lattes\ map of degree $d\geq2$. We recall that after a fixed finite extension $L/K$, there are an elliptic curve
$E/L$, a cyclic group $\Gamma\subseteq\Aut(E,O)$ of order
$r\in\{2,3,4,6\}$, a quotient map $\pi:E\to\P^1$ of degree
$r$, and an affine map
\[
 \psi(P)=a(P)+b,
 \qquad a\in\End(E),\quad\deg a=d,
 \qquad b\in E_{\tors},
\]
all defined over $L$, such that $f\circ\pi=\pi\circ\psi$.
Moreover,
\begin{equation}\label{eq:lbg-lattes-height}
 \hh f(\pi(P))=r\widehat h_E(P)
 \qquad(P\in E(\Kbar)).
\end{equation}

In fact, the fiber of $\pi$ over $\pi(O)$ consists of $O$ with multiplicity
$r$, because $O$ is fixed by all of $\Gamma$. Therefore
\[
 \pi^*\mathcal O_{\P^1}(1)\simeq\mathcal O_E(rO),
 \qquad h(\pi(P))=r\widehat h_E(P)+O(1).
\]
Write $\psi^{\circ n}(P)=a^n(P)+b_n$, with $b_n$ torsion.
The N\'eron--Tate height satisfies
\[
 \widehat h_E(\psi^{\circ n}(P))
   =\widehat h_E(a^n(P))=d^n\widehat h_E(P).
\]
Using the semiconjugacy and dividing the height comparison by
$d^n$, we obtain
\[
 \hh f(\pi(P))
   =\lim_{n\to\infty}d^{-n}h(\pi(\psi^{\circ n}(P)))
   =r\widehat h_E(P).
\]

If $E$ is non-CM, then $\End_{\Kbar}(E)=\Z$ and
$\Aut(E,O)=\{\pm1\}$. Consequently $r=2$, $a=[m]$, and
$2b=0$. Any two degree-two quotients by $[-1]$ differ by a
M\"obius transformation, so that
\[
 \pi=M_0\circ x
 \quad\text{for some }M_0\in\PGL_2(L),
\]
after enlarging $L$ once if necessary. For the sake of the proof,
replacing $K$ by a finite extension is harmless, so we will work over 
$L$ as above.

Now $r=2$ and $\pi=M_0\circ x$, with
$M_0\in\PGL_2(L)$. If $\alpha\in\P^1(L^{\mathrm{ab}})$ and
$\pi(P)=\alpha$, then
\[
 x(P)=M_0^{-1}(\alpha)\in\P^1(L^{\mathrm{ab}}).
\]
Proposition~\ref{prop:lbg-elliptic-coordinate}, applied over $L$,
and \eqref{eq:lbg-lattes-height} give
\[
 \hh f(\alpha)>0\quad\Longrightarrow\quad
 \hh f(\alpha)=2\widehat h_E(P)\geq2c(E,L).
\]
There are also only finitely many such $\alpha$ of height zero:
their lifts are torsion by \eqref{eq:lbg-lattes-height}, and the
last assertion of Proposition~\ref{prop:lbg-elliptic-coordinate}
applies to their $x$-coordinates. Hence the points in
$\P^1(L^{\mathrm{ab}})$ of height less than $2c(E,L)$ form a
finite set.
\end{proof}

We remark that Theorem \ref{thm:small} generalizes results of Ostafe \cite{Ost17} and Chen \cite{Che18}.

\begin{corollary}\label{thm:targets}
Let $f\in K(z)$ be non-strongly exceptional of degree at least $2$, and let
$\mathcal B\subseteq\P^1(\Kbar)$ have bounded Weil height.
Then there is an $N\geq1$ such that
$$
 \{P\in\P^1(\Kab):\iter f n(P)\in\mathcal B
                         \text{ for some }n\geq N\}
$$
is finite.
\end{corollary}
\begin{proof}
Assume $\mathcal B\ne\varnothing$. By~\eqref{eq:canonical},
$B:=\sup_{Q\in\mathcal B}\hh f(Q)<\infty$. Take $\varepsilon$
as in Theorem~\ref{thm:small} and choose $N$ with $d^{-N}B<\varepsilon$.
If $\iter f n(P)\in\mathcal B$ and $n\geq N$, then
$\hh f(P)\leq d^{-N}B<\varepsilon$, and Theorem \ref{thm:small} gives finiteness.
\end{proof}

Following~\cite[Definition~1.1]{Che18}, let $M\subseteq \overline{K}$ be a field and call $f$
\emph{$\mathcal B$-avoiding over $M$} if
$f^{-1}(\mathcal B)\cap\P^1(M)$ is finite, and \emph{strongly
$\mathcal B$-avoiding over $M$} if
$\bigcup_{n\geq1}f^{-n}(\mathcal B)\cap\P^1(M)$ is finite.
\begin{corollary}\label{cor:avoid}
Let $f\in K(z)$ be non-strongly exceptional, let
$K\subseteq M\subseteq\Kab$, and let $\mathcal B$ have bounded
Weil height. Then $f$ is $\mathcal B$-avoiding over $M$ if and only
if it is strongly $\mathcal B$-avoiding over $M$.
\end{corollary}
\begin{proof}
Only the forward implication requires proof. Put
$S_1=f^{-1}(\mathcal B)\cap\P^1(M)$. If $P\in\P^1(M)$ and
$\iter f n(P)\in\mathcal B$, then $\iter f{n-1}(P)\in S_1$.
Thus the level-$n$ set is contained in the finite set
$f^{-(n-1)}(S_1)$. Corollary~\ref{thm:targets} controls all levels
$n\geq N$ at once; the remaining levels form a finite union.
\end{proof}

In particular, when $\mathcal B$ is the set of roots of unity, Corollary \ref{cor:avoid} recovers \cite[Theorem 1.2]{Ost17} with the much weaker hypothesis that $f$ is not strongly exceptional.

For $A\geq1$, let $\mathcal I_A$ be the algebraic integers in
$\Kbar$ of house at most $A$, where the house of $\beta$ is
$\max_{\tau:\Q(\beta)\hookrightarrow\C}\abs{\tau(\beta)}$.
Then $\mathcal I_1=\{0\}\cup\mu_\infty$ and $h(\beta)\leq\log A$
for $\beta\in\mathcal I_A$.
\begin{corollary}\label{cor:house}
Let $f\in K(z)$ be non-strongly exceptional and let $M=\Kab$ or $M=\Kcyc$, the cyclotomic closure of $K$. For $\mathcal B=\mu_\infty$, or $\mathcal B=\mathcal I_A$ with $A\geq1$, one-step $\mathcal B$-avoidance over $M$ implies strong
$\mathcal B$-avoidance over $M$.
\end{corollary}
\begin{proof}
Both targets have bounded Weil height, so apply
Corollary~\ref{cor:avoid}.
\end{proof}

For rational maps that are not strongly exceptional, Corollaries \ref{cor:avoid} and \ref{cor:house} extend \cite[Theorem 1.5]{Che18} by removing the degree-gap hypothesis, allowing maximal abelian extensions, and permitting arbitrary target sets of bounded height.

\section{Abelian points for CM Latt\`es maps}\label{sec:elliptic}
In this section we prove Theorems \ref{thm:general-lattes} and \ref{general_ap}. We start by showing that for Latt\`es maps, wandering points admit finitely many abelian points in preorbits.

\begin{proposition}\label{prop:lattes-wandering}
Let $f\in K(z)$ be a Latt\`es map and
$\alpha\in\P^1(K)\setminus\PrePer(f)$. Then
$f^{-\infty}(\alpha)\cap\P^1(\Kab)$ is finite.
\end{proposition}
\begin{proof}
Choose a finite extension $L/K$ over which the presentation
\eqref{eq:lattes}, the elements of $\Gamma$, and the points of
$\pi^{-1}(\alpha)$ are defined. We have $\Kab\subseteq L^{\mathrm{ab}}$.
Let $Q\in\P^1(\Kab)$ satisfy $f^n(Q)=\alpha$, and choose $P$ with
$\pi(P)=Q$. Then $P$ is nontorsion and
\[
 \psi^{\circ n}(P)=a^nP+b_n\in\pi^{-1}(\alpha)\subseteq E(L).
\]
For $\sigma\in G_{L^{\mathrm{ab}}}$, the quotient property gives
$\sigma P=\zeta P$ for some $\zeta\in\Gamma$. Applying $\sigma$ to
the displayed equality gives $a^n(\zeta-1)P=0$. If $\zeta\ne1$,
this is a nonzero isogeny, contradicting that $P$ is nontorsion.
Thus $P\in E(L^{\mathrm{ab}})$.

The point $b_n$ is torsion. Hence
\[
 d^n\widehat h_E(P)
 =\widehat h_E(\psi^{\circ n}(P))
 \le M:=\max_{T\in\pi^{-1}(\alpha)}\widehat h_E(T).
\]
Baker--Silverman~\cite[Theorem~0.1]{BS04} gives
$\widehat h_E(P)\ge c_{E,L}>0$. Therefore $d^n c_{E,L}\le M$,
so $n$ is bounded. Each level of the backward orbit is finite.
\end{proof}

\begin{lemma}\label{lem:source-descent}
Every Latt\`es map $f\in K(z)$ admits an elliptic curve $E/K$,
an endomorphism $a\in\End_{\Kbar}(E)$ of degree $\deg f$, and a
finite morphism $\theta:E_{\Kbar}\to\P^1_{\Kbar}$ satisfying
$$
 f\circ\theta=\theta\circ a.
$$
\end{lemma}
\begin{proof}
Take a normalized presentation $f\circ\pi=\pi\circ(a_0+b)$ on
$E_0/\Kbar$. Since $b$ is torsion and $1-a_0$ is an isogeny,
every solution of $(1-a_0)c=b$ is torsion. The map
$\theta_0(P)=\pi(P+c)$ satisfies
$f\circ\theta_0=\theta_0\circ a_0$.

It remains to show $j(E_0)\in K$. If $|\Gamma|=3$, $4$, or $6$,
then $j(E_0)\in\{0,1728\}$. If $|\Gamma|=2$, the four branch
points of $\pi$ form the postcritical set of $f$
\cite[Lemma~3.4]{Mil06}, hence are $G_K$-stable. For any ordering
of these points, let $\lambda$ be their cross-ratio. Then
\[
 j(E_0)=256\frac{(1-\lambda+\lambda^2)^3}
                       {\lambda^2(1-\lambda)^2}.
\]
This expression is invariant under permutations of the four branch
points, so it belongs to $K$. Choose $E/K$ with this $j$-invariant
and an isomorphism $\varphi:E_{\Kbar}\to E_0$
\cite[Chapter~III, Sections~1 and~10]{Sil09}. Set
$a=\varphi^{-1}a_0\varphi$ and $\theta=\theta_0\varphi$.
\end{proof}
The next proposition is the analogue of \eqref{finitely_many_pts} for CM elliptic curves.
\begin{proposition}\label{prop:bounded-degree-torsion}
Let $E/K$ have geometric CM by an order $\mathcal O$ in an imaginary
quadratic field $F\not\subseteq K$. For every $D\ge1$, the set
$$
 \{P\in E_{\tors}:[\Kab(P):\Kab]\le D\}
$$
is finite.
\end{proposition}
\begin{proof}
Let $\rho$ be the adelic Galois representation, and put
$G=\rho(G_K)$. For each prime $\ell$, take the integral CM Cartan
\[
 C_\ell=(F\otimes\Q_\ell)^\times
             \cap\Aut_{\Z_\ell}(T_\ell E),\qquad
 C=\prod_\ell C_\ell.
\]
CM theory gives
$H:=\rho(G_{KF})=G\cap C$ of index two in $G$, open in $C$;
the other coset acts on $C$ by $c\mapsto\bar c$
\cite[Theorem~6.6]{Lom17}. Taking the Cartans on the actual Tate
lattices accommodates nonmaximal CM orders; away from the conductor,
$C_\ell=(\mathcal O\otimes\Z_\ell)^\times$.

Choose $\tau\in G_K\setminus G_{KF}$. For
$\sigma\in G_{KF}$ and $c=\rho(\sigma)$,
$$
 \rho([\tau,\sigma])=\bar c\,c^{-1}.
$$
Every commutator lies in $G_{\Kab}$, so
$$
 \delta(H)\subseteq\rho(G_{\Kab}),
 \qquad \delta(c)=\bar c\,c^{-1}.
$$
Choose a product $U=\prod_\ell U_\ell\subseteq H$, where each
$U_\ell$ is open in $C_\ell$ and $U_\ell=C_\ell$ for all but
finitely many $\ell$. With $V_\ell=\delta(U_\ell)$, we obtain
$$
 V:=\prod_\ell V_\ell\subseteq\rho(G_{\Kab}).
$$
If $[\Kab(P):\Kab]\le D$, write $P=\sum_\ell P_\ell$ for its
primary decomposition. Then
\begin{equation}\label{eq:local-orbit-upper}
 \#(V_\ell P_\ell)\le D.
\end{equation}

Let $S$ contain $2$, the primes ramified in $F$, those dividing the
conductor of $\mathcal O$, and those with $U_\ell\ne C_\ell$.
For $\ell\notin S$, the reduction of $V_\ell$ is
\[
 \begin{cases}
 \{(u,u^{-1}):u\in\mathbf F_\ell^\times\},
       &\ell\text{ split in }F,\\
 \ker(N:\mathbf F_{\ell^2}^\times\to\mathbf F_\ell^\times),
       &\ell\text{ inert in }F.
 \end{cases}
\]
These groups have order $\ell-1$ and $\ell+1$, respectively, and
act freely on $E[\ell]\setminus\{0\}$. If $P_\ell\ne0$, a
multiple of $P_\ell$ has order $\ell$, so
$D\ge\#(V_\ell P_\ell)\ge\ell-1$. Only finitely many primes can
therefore divide the order of $P$.

Fix one such prime $\ell$. The map $c\mapsto\bar c/c$ has
one-dimensional image, so $V_\ell$ is open in the local norm-one
torus and contains an element $v_\ell$ of infinite order.
By~\eqref{eq:local-orbit-upper}, two of
$P_\ell,v_\ell P_\ell,\ldots,v_\ell^D P_\ell$ coincide. Thus
\begin{equation}\label{eq:vk-kills}
 (v_\ell^k-1)P_\ell=0\qquad\text{for some }1\le k\le D.
\end{equation}
The endomorphism $v_\ell^k-1$ is invertible on
$T_\ell E\otimes\Q_\ell$: this follows from $v_\ell^k\ne1$
when $F\otimes\Q_\ell$ is a field; in the split case,
$v_\ell=(u,u^{-1})$ and neither component of $v_\ell^k-1$ vanishes.
Its kernel on $E[\ell^\infty]$ is therefore finite.
Equation~\eqref{eq:vk-kills} leaves finitely many possibilities for
$P_\ell$. This bounds every primary component and proves the claim.
\end{proof}

\begin{proof}[Proof of Theorem~\ref{thm:abelian_lattes}]
Since $\alpha$ is nonexceptional, $(ii)\Rightarrow(i)$.

Assume $(i)$. Proposition~\ref{prop:lattes-wandering} gives
\begin{equation}\label{eq:alpha-preperiodic}
 \alpha\in\PrePer(f).
\end{equation}
Its backward orbit is then preperiodic, so
Proposition~\ref{prop:torsion} excludes a non-CM presentation.
Lemma~\ref{lem:source-descent} gives a CM elliptic curve $E/K$,
an endomorphism $a$, and a finite morphism $\theta$ defined over
some finite extension $L/K$, with $f\circ\theta=\theta\circ a$.
Write $F=\End_{\Kbar}(E)\otimes\Q$.

Choose distinct $\beta_i\in f^{-\infty}(\alpha)\cap\P^1(\Kab)$ and
lifts $P_i\in\theta^{-1}(\beta_i)$. The finite-semiconjugacy
criterion above gives
$$
 P_i\in E_{\tors}.
$$
Each fiber is defined over $L\Kab$ and has at most $\deg\theta$
points. Thus
$$
 [\Kab(P_i):\Kab]\le(\deg\theta)[L:K].
$$
The $P_i$ are distinct, so Proposition~\ref{prop:bounded-degree-torsion}
forces $F\subseteq K$.

All endomorphisms of $E$ are now defined over $K$. The adelic image
lies in the abelian CM Cartan~\cite[Theorem~6.6(1)]{Lom17}, hence
\begin{equation}\label{eq:all-torsion-ab}
 E_{\tors}\subseteq E(\Kab).
\end{equation}
For $\sigma\in G_{\Kab}$, both $P_i$ and $\beta_i$ are fixed, so
\[
 \theta^\sigma(P_i)=\sigma(\theta(P_i))
                   =\beta_i=\theta(P_i).
\]
The morphisms $\theta^\sigma$ and $\theta$ agree at infinitely many
points of an integral curve, so they are equal. Thus $\theta$ is
defined over $\Kab$. Equation~\eqref{eq:alpha-preperiodic} gives
$\alpha\in\theta(E_{\tors})$, proving $(iii)$.

Finally, assume $(iii)$ and let $\beta\in f^{-n}(\alpha)$.
The point $\alpha$ is preperiodic, hence so is $\beta$.
Every lift $P\in\theta^{-1}(\beta)$ is therefore torsion.
By~\eqref{eq:all-torsion-ab}, $P\in E(\Kab)$; since $\theta$
is defined over $\Kab$, we have $\beta\in\P^1(\Kab)$.
This proves $(ii)$.
\end{proof}

\begin{proof}[Proof of Theorem~\ref{general_ap}]
Suppose that $f^{-\infty}(\alpha)\cap\P^1(\Kab)$ is infinite.
Choose distinct points $x_i$ in this intersection with
$\iter f{n_i}(x_i)=\alpha$. Since each finite union of levels is
finite, we may assume that $n_i\to\infty$. Then
\[
 \hh f(x_i)=d^{-n_i}\hh f(\alpha)\longrightarrow0.
\]
Theorem~\ref{thm:small} implies that $f$ is exceptional.
If $f$ is a Latt\`es map, Theorem~\ref{thm:abelian_lattes}
gives~\textup{(iii)}.

Otherwise, choose $m\in\operatorname{PGL}_2(\Kbar)$ such that
$g=m\circ f\circ m^{-1}$ belongs to
$\{z^{\pm d},\pm C_d\}$, and put $\beta=m(\alpha)$.
Let $L/K$ be a finite extension over which $m$ is defined.
The points $m(x_i)$ lie in
$g^{-\infty}(\beta)\cap\P^1(L\Kab)$, hence in
$g^{-\infty}(\beta)\cap\P^1(L^{\mathrm{ab}})$.
The point $\beta$ is nonexceptional for $g$, so it is finite.
By~\cite[Lemma~3.3(1) and Remark~2.2]{FOZ24},
$\beta=\zeta$ if $g=z^{\pm d}$, and
$\beta=\zeta+\zeta^{-1}$ if $g=\pm C_d$, for a root of unity
$\zeta$. In either case, the whole backward orbit of $\beta$
is cyclotomic over~$\mathbf Q$.

Choose three distinct $x_i$. These points and their images under
$m$ belong to $\P^1(\Kab)$. A M\"obius transformation is
determined by its values at three distinct points, so
$m\in\operatorname{PGL}_2(\Kab)$. This proves
\textup{(i)} or~\textup{(ii)}.

Conversely, the backward orbits of the standard pairs in
\textup{(i)} and~\textup{(ii)} are cyclotomic, and conjugacy over
$\Kab$ preserves $\P^1(\Kab)$. In case~\textup{(iii)},
Theorem~\ref{thm:abelian_lattes} gives
$f^{-\infty}(\alpha)\subseteq\P^1(\Kab)$.
In each case, the orbit is infinite because $\alpha$ is
nonexceptional.
\end{proof}

\end{document}